\documentclass[10pt]{amsart}
\usepackage{amsthm}
\usepackage{amsfonts}
\usepackage{amsmath}
\usepackage{amssymb}
\usepackage{mathrsfs}
\title{Euler's Criterion of  prime order in PID case}
\author{Jagmohan Tanti}
\date{}
\newcommand{\vpmod}{\!\!\!\pmod}
\makeatletter
\newcommand{\least}{\let\CS=\@currsize\renewcommand{\baselinestretch}{.9}\tiny\CS}
\renewcommand\baselinestretch{1.1} 

\newcommand{\ncom}{\newcommand}
\ncom{\ul}{\underline}
\ncom{\ol}{\overline}
\ncom{\bq}{\begin{equation}}
\ncom{\eq}{\end{equation}}
\ncom{\beqn}{\begin{eqnarray*}}
\ncom{\eeqn}{\end{eqnarray*}}
\ncom{\beq}{\begin{eqnarray}}
\ncom{\eeq}{\end{eqnarray}}
\ncom{\nno}{\nonumber}
\ncom{\rar}{\rightarrow}
\ncom{\Rar}{\Rightarrow}
\ncom{\noin}{\noindent}
\ncom{\bc}{\begin{centre}}
\ncom{\ec}{\end{centre}}
\ncom{\sz}{\scriptsize}
\ncom{\rf}{\ref}
\ncom{\sgm}{\sigma}
\ncom{\Sgm}{\Sigma}
\ncom{\dt}{\delta}
\ncom{\Dt}{Delta}
\ncom{\lmd}{\lambda}
\ncom{\Lmd}{\Lambda}
\ncom{\eps}{\epsilon}
\ncom{\pcc}{\stackrel{P}{>}}
\ncom{\dist}{{\rm\,dist}}
\ncom{\sspan}{{\rm\,span}}
\ncom{\re}{{\rm Re\,}}
\ncom{\im}{{\rm Im\,}}
\ncom{\sgn}{{\rm sgn\,}}
\ncom{\ba}{\begin{array}}
\ncom{\ea}{\end{array}}
\ncom{\eop}{\hfill{{\rule{2.5mm}{2.5mm}}}}
\ncom{\eoe}{\hfill{{\rule{1.5mm}{1.5mm}}}}
\ncom{\eof}{\hfill{{\rule{1.5mm}{1.5mm}}}}
\ncom{\hone}{\mbox{\hspace{1em}}}
\ncom{\htwo}{\mbox{\hspace{2em}}}
\ncom{\hthree}{\mbox{\hspace{3em}}}
\ncom{\hfour}{\mbox{\hspace{4em}}}
\ncom{\vone}{\vskip 2ex}
\ncom{\vtwo}{\vskip 4ex}
\ncom{\vonee}{\vskip 1.5ex}
\ncom{\vthree}{\vskip 6ex}
\ncom{\vfour}{\vspace*{8ex}}
\ncom{\norm}{\|\;\;\|}
\ncom{\integ}[4]{\int_{#1}^{#2}\,{#3}\,d{#4}}
\ncom{\inp}[2]{\langle{#1},\,{#2} \rangle}
\ncom{\Inp}[2]{\Langle{#1},\,{#2} \Langle}
\ncom{\vspan}[1]{{{\rm\,span}\#1 \}}}
\ncom{\dm}[1]{\displaystyle {#1}}

\newtheorem{theorem}{\bf Theorem}[section]
\newtheorem{lemma}[theorem]{\bf Lemma}

\newtheoremstyle
	{remarkstyle}
	{}
	{11pt}
	{}
	{}
	{\bfseries}
	{:}
	{     }
	{\thmname{#1} \thmnumber{#2} }

\theoremstyle{remarkstyle}

\newtheorem{remark}[theorem]{\bf Remark}
\newtheorem{definition}[theorem]{\bf Definition}

\newtheorem{conjecture}[theorem]{\bf Conjecture}
\begin{document}
\maketitle
\renewcommand{\thefootnote}{}
\footnote{ \noindent\textbf{} \vskip0pt
\textbf{Mathematics Subject Classification:}
Primary: 11T22, 11T24.
}
\begin{abstract}
Let $l\geq2$ be a prime, $p$ a prime $\equiv 1 \vpmod{l}$ and $\gamma$ a primitive root $\vpmod{p}$. If an integer $D$ with $(p,D)=1$, is an 
$l^{th}$ power nonresidue$\pmod{p}$ then 
$D^{(p-1)/l}$ is an $l^{th}$ root of unity $\alpha(\not \equiv 1)\,\vpmod{p}$. Euler's criterion of order $l$ $\pmod{p}$ studies the explicit 
conditions when 
$D^{(p-1)/l}\equiv \gamma^{(p-1)/l}\vpmod{p}$, i. e., when $Ind_\gamma D\equiv1\vpmod{l}$.  
In this paper we establish the Euler's criterion of order $l$ when the ring of integers in the cyclotomic extension of $\mathbb{Q}$ of order $l$ is a PID.
Conditions are obtained in terms of Jacobi sums of order $l$.
\end{abstract}

\section{Introduction} Let $e$ be an integer $\geq2$, and $p$ a prime $\equiv 1 \vpmod{e}$. Euler's criterion states that for $D\in \mathbb{Z}$ and 
$(D,p)=1$,
\begin{eqnarray}
D^{\frac{p-1}{e}}\equiv 1\vpmod{p}
\end{eqnarray}
if and only if $D$ is an $e$th power residue $\vpmod{p}$. If $D$ is not an $e$-th power $\vpmod{p}$, one has
\begin{eqnarray}
D^{\frac{p-1}{e}}\equiv \alpha\vpmod{p}
\end{eqnarray}
for some $e$-th root $\alpha(\not \equiv 1)$ of unity $\vpmod {p}$.

As an example for $p\equiv1\pmod{3}$ one considers the integer solutions $L$ and $M$ of a quadratic partition (Gauss system)

$$4p=L^2+27M^2,\,\,L\equiv1\pmod{3}$$ and Euler's criterion for $e=3$ is given by following conditions:\\
For $D\in\mathbb{Z}$ coprime to $p$, we have\\
$$D^\frac{p-1}{3}\equiv\left\{\begin{array}{ll}
                              1 & \textrm{if $D$ is a cubic residue $\pmod{p}$}\\
                              \frac{L\pm9M}{L\mp9M} & \textrm{otherwise}.
                             \end{array}\right.$$
Here Jacobi sum is $J(1,1)=\frac{1}{2}(L+3M)+3M\omega$, where $\omega=e^\frac{2\pi i}{3}$.\\

A problem concerning Euler's criterion is to determine for a given $e$-th power nonresidue $D\vpmod{p}$ an $e$-th root of unity $\alpha \vpmod{p}$
in terms of the solution of the corresponding Diophantine system so that $ D^{\frac{p-1}{e}}\equiv \alpha\vpmod{p}$. One may also consider 
the problem of obtaining congruence conditions on the solutions of the corresponding system so that $D^{\frac{p-1}{e}}\equiv 1\vpmod{p}$,
i.e., $D$ is an $e$-th power residue $\vpmod{p}$. In the literature this problem has been discussed for some small values of $e$ with different approaches.
Some times people use certain quadratic partitions of primes to obtain the concerned conditions. When $e=2$, the result is well known in terms of 
Quadratic reciprocity Laws, Gauss establishes for $e=4$,  Western and Lehmer establishes for $e=8$, and for $e=16$ and $32$ it has been established by
Hudson and Williams \cite{Hw}. Lehmer solved this for $e=3$ and $D=2$ \cite{Le} and Williams for $e=3$ and every $D\in \mathbb{Z}$ with
explicit results when $D$ a prime $\leq 19$ \cite{Wi1}. Again Lehmer considered this problem for $e=5$ \cite{Le},  derived (cubic) expressions for fifth roots of unity $\vpmod{p}$ from the solutions of (4). She established Euler's criterion
for $D=2$ and $D=4$. Williams \cite{Wi2} used the same expressions for fifth roots of unity to solve the problem for every $D\in \mathbb{Z}$ with explicit 
results for $D= 3,\, 5$. Later it was found by Katre and Rajwade \cite{Kr1} that these expressions of fifth roots of unity are not always 
well defined. They obtained correct expressions for them and solved the problem of Euler's criterion for $e=5$ and every $D\in \mathbb{Z}$ 
with explicit results for $\displaystyle{2, 3, 5, 7}$. Lehmer used Jacobsthal sums to derive an expression for a fifth root 
of unity $\vpmod{p}$ whereas Katre and Rajwade used Jacobi sums. We considered the case $e=7$ \cite{Tk} and $e=11$ \cite{KT} for Euler's criterion with 
explicit results for $\displaystyle{2, 3, 5, 7}$ and $\displaystyle{2, 7, 11}$ respectively. 

One may encounter the level of complicacy for large values of $e$, even for $e=7,\,\,11$ while dealing with the concerned available quadratic
partitions. However such type of quadratic partitions are not seen for 
$e=13,\,\, 17$ etc but Jacobi sums of order $e$ exists and is easier to handle with certain basic concepts of Cyclotomic fields.

\vspace{0.2cm}

In this paper for $e=l$ a prime and $\zeta_l=e^\frac{2\pi i}{l}$,  establish the Euler's criterion for $l^{th}$ power nonresidues $\pmod p$, 
when $\mathbb{Z}[\zeta_l]$ is a PID and have given necessary 
and sufficient conditions for an integer $D$ coprime to $p$ to satisfy one of the equations (1) and (2). This paper is a sort of unifying the earlier 
published works e. g., $\displaystyle l=3, 5, 7, 11$ etc along this line.

\section{Preliminaries}

Let $l$ be a prime $\geq3$, $p$  a prime $\equiv1\vpmod{l}$ and $\zeta_l = \textrm{exp}(2\pi i/l)$. Then the ring of integers of the cyclotomic field 
$\mathbb{Q}(\zeta_l)$ is $\mathbb{Z}(\zeta_l)$ with $\{\displaystyle{\zeta_l, \zeta_l^2,\cdots, \zeta_l^{l-2}}\}$ as an integral basis. 
$\pm\zeta_l^i$, $0\leq i \leq {l-1}$ are the only roots of unity in $\mathbb{Z}[\zeta_l]$. The group of units in $\mathbb{Z}[\zeta_l]$ is 
\{$\pm{\zeta_l}^i\Pi_a\left(\zeta_l^{(1-a)/2}\frac{1-\zeta_l^a}{1-\zeta_l}\right)^{j_a}$\,:\, $\displaystyle{1<a<\frac{l}{2}}$, $(a,l)=1$,
$\displaystyle{i, j_a}\in \mathbb{Z}$, $0\leq i \leq l-1$\} \cite{Wa}. It is known that $1-\zeta_l$ is a prime element in $\mathbb{Z}[\zeta_l]$ and 
$l = \Pi_{i=1}^{l-1}(1-\zeta_l^i)$.

\vspace{0.2cm}

For $\gamma\in\mathbb{Z}$ a primitive root $\pmod{p}$, $\alpha=\gamma^\frac{p-1}{l}$ and $\phi_l(x)=1+x+\cdots+x^{l-1}\in\mathbb{Z}[x]$
the cyclotomic polynomial of order $l$ we have, $\phi_l(x)\equiv\Pi_{i=1}^{l-1}(x-\alpha^i)\pmod{p}$. Therefore we have 
$<p>=\Pi_{i=1}^{l-1}(p,\,\,\zeta_l-\alpha^i)$ a prime ideals factorization of $p$ in $\mathbb{Z}[\zeta_l]$. Let us denote $\mathcal{P}_i=(p,\,\,\zeta_l-\alpha^i)$
and $\sigma_i\in G(\mathbb{Q}(\zeta_l)/\mathbb{Q})$, $\sigma_i(\zeta_l)=\zeta_l^i$ for $1\leq i\leq l-1$, then it is easy to see that for 
$1\leq k\leq l-1$ we have $\mathcal{P}_k=(p,\,\,\zeta_l^{k^{-1}}-\alpha)=\sigma_{k^{-1}}(\mathcal{P}_1)$.

\vskip2mm

We define the character $\chi_l$ on 
$(\mathbb{Z}/p\mathbb{Z})^*$ by $\chi_l(\gamma)=\zeta_l.$ Define $$J(i,j)_l=\sum_{-1\not=v\in \mathbb{F}_p^*} {\chi_l^i(v)}{\chi_l^j(v+1)},$$
where $\mathbb{F}_p$ denotes a field of size $p$. $J(i,j)_l$ is called a Jacobi sum of order $l$. To know more about Jacobi sums one can refer to the 
book \cite{Be}.
\vskip2mm
Now by Stickelberger theorem \cite{Wa} if $\psi\in\mathbb{Z}[\zeta_l]$ such that $<\psi>=\Pi_{i=1}^{\frac{l-1}{2}}\mathcal{P}_1^{\sigma_i^{-1}}$, 
then $\psi$ is an associate of the Jacobi sum $J_l(1,1)$ of order $l$ i. e., $J_l(1,1)=u\psi$ for some unit $u\in\mathbb{Z}[\zeta_l]$. 
\vskip2mm
\begin{remark}
Here $|\psi|^2=\psi\overline{\psi}=\psi\sigma_{l-1}(\psi)= (\Pi_{i=1}^{\frac{l-1}{2}}\mathcal{P}_1^{\sigma_i^{-1}})
(\Pi_{i=1}^{\frac{l-1}{2}}\mathcal{P}_1^{\sigma_{\frac{l-1}{2}+i}^{-1}})=p$. So $|\psi|=\sqrt{p}=|J_l(1,1)|$.
\label{samemod}
\end{remark}

\section{Some Lemmas}
\begin{lemma}
 $J_l(1,1)\equiv -1 \vpmod {(1-\zeta_l)^2}$.
\end{lemma}
\begin{proof}
 See \cite{Par}.
\end{proof}
\begin{lemma}
Let $\displaystyle \alpha, \beta\in\mathbb{Z}[\zeta_l]$ both prime to $1-\zeta_l$ and satisfying 
(i) $<\alpha>=<\beta>$, (ii) $|\alpha|=|\beta|$, (iii) $\alpha\equiv\beta\pmod{(1-\zeta_l)^2}$ then $\alpha=\beta$. 
\label{samevalue}
\end{lemma}
\begin{proof}
See \cite{Par}. 
\end{proof}
\begin{lemma}
If $\alpha \in \mathbb{Z}[\zeta_{l}]$ is such that $\alpha \not\equiv 0 \pmod{(1-\zeta_{l})}$, then $\alpha$ possesses an associate $\beta$ such 
 that $\beta \equiv -1 \pmod{(1-\zeta_{l})^2}.$ 
\end{lemma}
\begin{proof}
Let $\alpha=a_1\zeta_l+a_2{\zeta_l}^2\cdots+a_{l-1}{\zeta_l}^{l-1}$. Since for 
$f(x)\in\mathbb{Z}[x]$, $f(\zeta_l)\equiv f(1)-f'(1)(1-\zeta_l)\pmod{(1-\zeta_l)^2}$ so $\alpha\equiv b-c(1-\zeta_l)\pmod{(1-\zeta_l)^2}$, 
where $b=a_1+a_2+\cdots+a_{l-1}$ and $c=a_1+2a_2+\cdots+(l-1)a_{l-1}$. As 
$\alpha \not\equiv 0 \pmod{(1-\zeta_l)}$, so $b\not\equiv 0 \vpmod{l}$. 
Now let $a$ be a primitive root $\pmod{l}$, then there exists a unique $0\leq d \leq l-2$ such that ${a^d}b \equiv -1 \pmod{l}$. 
Thus we have
\begin{eqnarray*}
{\zeta_l}^{c{a^d}}\alpha&\equiv& (1-(1-\zeta_{l}))^{ca^d}(b-c(1-\zeta_l))\pmod{(1-\zeta_l)^2}\\
&\equiv&(1-ca^d(1-\zeta_l))(b-c(1-\zeta_l))\pmod{(1-\zeta_l)^2}\\
&\equiv&b-(c+ca^d b)(1-\zeta_l)\pmod{(1-\zeta_l)^2}\\
&\equiv&b\pmod{(1-\zeta_l)^2}.
\end{eqnarray*}
Now choose a unit $u=\zeta_l^\frac{1-a}{2}\frac{1-\zeta_l^a}{1-\zeta_l}$, then we see that $u\equiv a\pmod{(1-\zeta_l)^2}$.
Let $\beta={\zeta_l}^{ca^d}u^d\alpha$, 
then $\beta\equiv bu^d \pmod{(1-\zeta_l)^2}\equiv a^db\equiv-1\pmod{(1-\zeta_l)^2}$.
\end{proof}
\begin{lemma}
If $\mathbb{Z}[\zeta_l]$ is a PID, then there exists $\mathcal{K}\in\mathbb{Z}[\zeta_l]$ such that 
$\mathcal{P}_1=<\mathcal{K}>$, $\mathcal{K}\equiv-1\pmod{(1-\zeta_l)^2}$ and $J(1,1)=(-1)^\frac{l+1}{2}\Pi_{i=1}^{\frac{l-1}{2}}\mathcal{K}_i^{-1}$.
\end{lemma}
\begin{proof}
As $\mathbb{Z}[\zeta_l]$ is a PID, there exists $K\in\mathbb{Z}[\zeta_l]$ such that $\mathcal{P}_1=<K>$. Also as $K$ and $1-\zeta_l$ are relatively prime,
$K$ possesses an associate $\mathcal{K}\in\mathbb{Z}[\zeta_l]$ such that $\mathcal{K}\equiv-1\pmod{(1-\zeta_l)^2}$ and so 
for $1\leq i\leq l-1$ we have $\mathcal{K}_i=\sigma_i(\mathcal{K})\equiv-1\pmod{(1-\zeta_l)^2}$.
\vskip2mm
Thus we have $<J_l(1,1)>=\Pi_{i=1}^\frac{l-1}{2}<\mathcal{K}_i^{-1}>=\left<\Pi_{i=1}^\frac{l-1}{2}\mathcal{K}_i^{-1}\right>$. 
Now let $\phi=(-1)^\frac{l+1}{2}\Pi_{i=1}^\frac{l-1}{2}\mathcal{K}_i^{-1}$, then we have
(i) $<\phi>=<J_l(1,1)>$, (ii) $\phi\equiv-1\equiv J_l(1,1)\pmod{(1-\zeta_l)^2}$ and 
(iii) $|\phi|^2=\phi\overline{\phi}=(-1)^{l+1}\Pi_{i=1}^{l-1}\mathcal{K}_i=N(\mathcal{K})=p$ and so $|\phi|=\sqrt{p}$. 

Thus by the lemma \ref{samevalue} we get $\phi=J_l(1,1)$ and hence $J(1,1)=(-1)^\frac{l+1}{2}\Pi_{i=1}^{\frac{l-1}{2}}\mathcal{K}_i^{-1}$.

\end{proof}

\section{Outline of the method}
\indent In this section we discus about the Euler's criterion for $l^{th}$ power residues and nonresidues $\pmod{p}$ for $p\equiv1\pmod{l}$ in the case 
when $\mathbb{Z}[\zeta_l]$ is a PID.\\

\begin{definition}
For $a\in\mathbb{Z}[\zeta_l]$ and $\pi$ a prime element in $\mathbb{Z}[\zeta_l]$ coprime to $l$, define the $l^{th}$ power residue symbol, 
$\left(\frac{a}{\pi}\right)_l$, as follows:\\
$\left(\frac{a}{\pi}\right)_l=\left\{\begin{array}{ll}
 0 & \textrm{if $\pi|a$},\\
 \zeta_l^i & \textrm{if $\pi\nmid a$ and $a^{\frac{N(\pi)-1}{l}}\equiv\zeta_l^i\pmod{\pi}$}.
                                     \end{array}\right.$
\end{definition}
 
\noindent Note that if $u$ is a unit of $\mathbb{Z}[\zeta_l]$, $\left(\frac{a}{u\pi}\right)_l=\left(\frac{a}{\pi}\right)_l$.

\vspace{.2cm}

\noindent$\bf{Properties}$
\begin{enumerate}
\item[(a)] $\left(a/\pi \right)_l=1$ if and only if $x^l\equiv a\,\pmod{\pi}$ is solvable in $\mathbb{Z}[\zeta_l]$.
\item[(b)] For all $a\in\mathbb{Z}[\zeta_l]$, $a^{\left(N(\pi)-1\right)/l}\equiv\left(a/\pi\right)_l\,\pmod{\pi}$.
\item[(c)] For $a$ and $b$ in $\mathbb{Z}[\zeta_l]$, $\left(ab/\pi\right)_l=\left(a/\pi\right)_l\left(b/\pi\right)_l$.
\item[(d)] If $a\equiv b\,\pmod{\pi}$ then $\left(a/\pi\right)_l=\left(b/\pi\right)_l$.
\item[(e)] For $\sigma\in\textrm{G}\left(\mathbb{Q}(\zeta_l)/\mathbb{Q}\right)$, $\left(\frac{a}{\pi}\right)^{\sigma}_l=\left(\frac{a^\sigma}{\pi^\sigma}\right)_l$.
\item[(f)] If $\pi_1,\cdots, \pi_k$ are coprime to $l$ in $\mathbb{Z}[\zeta_l]$ then one defines  $\left(a/{\pi_1\cdots\pi_k}\right)_l=\left(a/\pi_1\right)_l\cdots\left(a/\pi_k\right)_l$.
\item[(g)] If $\eta\in\mathbb{Z}[\zeta_l]$ is coprime to $l$, then $\left(\frac{ab}{\eta}\right)_l=\left(\frac{a}{\eta}\right)_l\left(\frac{b}{\eta}\right)_l$.
\end{enumerate}

\vspace{0.2cm}

\noindent \underline{Eisenstein Reciprocity Law} \cite{Ad}. Let $\theta \in \mathbb{Z}[\zeta_l],\,\,(\theta,l)=1,$ such that $\theta\pmod{(1-\zeta_l)^2}$ is congruent to a 
rational integer. Then for $a\in \mathbb{Z},\, (l, a)=1$, we have $(\theta, a)=1\implies\left(\frac{\theta}{a}\right)_l = \left(\frac{a}{\theta}\right)_l.$

\vspace{0.2cm}

\begin{conjecture}
For $l\geq 3$ a rational prime, the sum $\sum_{i=1}^\frac{l-1}{2}i^{-1}\not\equiv0\pmod{l}$.
\label{conject}
\end{conjecture}
\begin{theorem}
Let $\phi=J_l(1,1)$, $D\in\mathbb{Z}$ satisfying $(D,p)=1=(D,l)$ then\\
(i) $D$ is an $l^{th}$ power $\pmod{p}$ if and only if $\left(\frac{\phi}{D}\right)_l=1$,\\
(ii) $Ind_\gamma(D)\equiv1\pmod{l}$ if and only if $\left(\frac{\phi}{D}\right)_l=\zeta_l^{\sum_{i=1}^\frac{l-1}{2}i^{-1}}$.
\end{theorem}
\begin{proof}
\noindent(i) $D$ is an $l^{th}$ power $\pmod{p}$ iff $D^\frac{p-1}{l}\equiv1\pmod{p}$ iff $D^\frac{p-1}{l}\equiv1\pmod{\phi}$ and 
$D^\frac{p-1}{l}\equiv1\pmod{\overline{\phi}}$ iff $D^\frac{p-1}{l}\equiv1\pmod{\phi}$ 
(as $\left(\frac{D}{\overline{\phi}}\right)_l=\left(\frac{D}{\phi}\right)_l^{\sigma_{l-1}}=1$) iff 
$\left(\frac{D}{\phi}\right)_l=\left(\frac{\phi}{D}\right)_l$ (by Eisenstein's Reciprocity law).\\
\vskip2mm
\noindent(ii) $Ind_{\gamma}(D)\equiv1\pmod{l}$ iff $D^\frac{p-1}{l}\equiv\gamma^\frac{p-1}{l}=\alpha(say)\pmod{p}$ iff 
$D^\frac{p-1}{l}\equiv\alpha\pmod{\mathcal{K}}$ iff $\left(\frac{D}{\mathcal{K}}\right)_l=\zeta_l$ (as $\alpha-\zeta_l\equiv0\pmod{\mathcal{K}}$).
Now by Eisenstein's Reciprocity law, $\left(\frac{\phi}{D}\right)_l=\left(\frac{D}{\phi}\right)_l$,\\ 
so $\left(\frac{\phi}{D}\right)_l=\left(\frac{D}{(-1)^\frac{l+1}{2}\Pi_{i=1}^\frac{l-1}{2}\mathcal{K}_i^{-1}}\right)_l=
\left(\frac{D}{\Pi_{i=1}^\frac{l-1}{2}\mathcal{K}_i^{-1}}\right)_l=\Pi_{i=1}^\frac{l-1}{2}\left(\frac{D}{K}\right)_l^{\sigma_i^{-1}}=
\Pi_{i=1}^\frac{l-1}{2}\zeta_l^{\sigma_i^{-1}}=\Pi_{i=1}^\frac{l-1}{2}\zeta_l^{i^{-1}}=\zeta_l^{\sum_{i=1}^\frac{l-1}{2}i^{-1}}$.
\vskip2mm
Now for the converse part assume that $\left(\frac{\phi}{D}\right)_l=\zeta_l^{\sum_{i=1}^\frac{l-1}{2}i^{-1}}$. 
Also assume that for some $1\leq j\leq l-1$, $\left(\frac{\mathcal{K}}{D}\right)_l=\zeta_l^j$. Then again by Eisenstein's Reciprocity law we have 
$\left(\frac{\phi}{D}\right)_l=\left(\frac{D}{\phi}\right)_l=\Pi_{i=1}^\frac{l-1}{2}\left(\frac{D}{\mathcal{K}}\right)_l^{\sigma_i^{-1}}=
\Pi_{i=1}^\frac{l-1}{2}\left(\frac{\mathcal{K}}{D}\right)_l^{\sigma_i^{-1}}=\Pi_{i=1}^\frac{l-1}{2}(\zeta_l^j)^{\sigma_i^{-1}}=
\zeta_l^{j\sum_{i=1}^\frac{l-1}{2}i^{-1}}$. So $j\sum_{i=1}^\frac{l-1}{2}i^{-1}\equiv\sum_{i=1}^\frac{l-1}{2}i^{-1}\pmod{l}$ which implies that 
$(j-1)\sum_{i=1}^\frac{l-1}{2}i^{-1}\equiv0\pmod{l}$. Thus referring to the Conjecture \ref{conject} we have $j\equiv1\pmod{l}$ and hence 
$\left(\frac{\mathcal{K}}{D}\right)_l=\zeta_l$.
\end{proof}
From \cite{Tk} we have follwing two Lemma's.
\begin{lemma}
Let $p\equiv 1\vpmod{l}$ then for $a_h(n)$ as defined in \cite{Tk},

(i) $l$ is an $l$th power $\pmod{p}$ if and only if $$(l-1)(p-l+1)+\sum_{n=1}^{l-2}\sum_{h=1}^{l-1}a_h(n)(2h-l+1)
\equiv 0\vpmod{l^2}.$$

(ii) If $l$ is not an $l$th power $\vpmod{p}$, $\textrm{ind}_{\gamma}l\equiv1\vpmod{l}$ if and only if $$(l-1)(p-3l+1)+\sum_{n=1}^{l-2}\sum_{h=1}^{l-1}a_h(n)(2h-l+1)\equiv 0\vpmod{l^2}.$$
\end{lemma}

\begin{lemma}
Let $p\equiv 1 \pmod{l}$, then

(i) $2$ is an $l^{th}$ power $\pmod{p}$ if and only if $\sum_{i=1}^{l-1}a_i(1) \equiv 0 \pmod{2}.$\\

(ii) If $2$ is not an $l^{th}$ power $\pmod{p}$, $\textrm{Ind}_{\gamma}2\equiv1\vpmod{l}$ if and only if $a_{l-2}(1)\equiv1\pmod{2}.$
\end{lemma}

\noindent{\bf Conclusion and Future work}: This paper answers the question of Euler's criterion of prime order $l$, in terms of formulae 
involving a Jacobi sum of order $l$ in the case when the ring of integers in 
$\mathbb{Q}[\zeta_l]$ is a PID. It is also expected that same formulae is true in the case when the said ring of integers is not a PID and this is a Future 
scope in this line.
\vskip2mm
\noindent{\bf Acknowledgments}: The author would like to thank the ICAA-2017, organized by Department of mathematics, Savitribai Phule Pune University,
Pune, India where the idea of this paper work came into existence. Also he would like to express his gratitude towards Cental University of Jharkhand, Ranchi, India where 
the paper was finalized.

\vspace{5mm}
\begin{large}
\noindent\textsc{Jagmohan Tanti}\\
\end{large}
{\it Department of Mathematics, Central University of Jharkhand}, {\it Ranchi-835205, India}\\
{\it E-mail: jagmohan.t@gmail.com}
\end{document}